\documentclass[reqno]{amsart}%
\usepackage{amssymb,mathtools,url,hyperref,cite}

\allowdisplaybreaks[4]
\theoremstyle{plain}
\newtheorem{thm}{Theorem}

\newtheorem{lem}{Lemma}
\theoremstyle{remark}
\newtheorem{rem}{Remark}
\DeclareMathOperator{\td}{d}
\DeclareMathOperator{\te}{e}

\begin{document}
\title[Decreasing and complete monotonicity of two functions]
{Decreasing and complete monotonicity of two functions defined by three derivatives of a completely monotonic function involving the trigamma function}

\author[H.-P. Yin]{Hong-Ping Yin}
\address{College of Mathematical Science, Inner Mongolia Minzu University, Tongliao, Inner Mongolia, 028043, China}
\email{\href{mailto: H.-P. Yin <hongpingyin@qq.com>}{hongpingyin@qq.com}}
\urladdr{\url{https://orcid.org/0000-0001-7481-0194}}

\author[L.-X. Han]{Ling-Xiong Han}
\address{College of Mathematical Science, Inner Mongolia Minzu University, Tongliao, Inner Mongolia, 028043, China}
\email{\href{mailto: L.X. Han <hlx2980@163.com>}{hlx2980@163.com}, \href{mailto: L.X. Han <hanlingxiong@outlook.com>}{hanlingxiong@outlook.com}}
\urladdr{\url{https://orcid.org/0000-0003-1346-9943}}

\author[F. Qi]{Feng Qi*}
\address{School of Mathematics and Physics, Hulunbuir University, Hulunbuir, Inner Mongolia, 021008, China;
School of Mathematics and Informatics, Henan Polytechnic University, Jiaozuo, Henan, 454010, China;
Independent researcher, University Village, Dallas, TX 75252, USA}
\email{\href{mailto: F. Qi <qifeng618@gmail.com>}{qifeng618@gmail.com}}
\urladdr{\url{https://orcid.org/0000-0001-6239-2968}}


\begin{abstract}
In the paper, by convolution theorem of the Laplace transforms, a monotonicity rule for the ratio of two Laplace transforms, Bernstein's theorem for completely monotonic functions, and other analytic techniques, the authors verify decreasing property of a ratio between three derivatives of a function involving trigamma function and find necessary and sufficient conditions for a function defined by three derivatives of a function involving trigamma function to be completely monotonic. These results confirm previous guesses posed by Qi and generalize corresponding known conclusions.
\end{abstract}



\thanks{*Corresponding author}

\thanks{This paper was typeset using \AmS-\LaTeX}

\maketitle

\section{Introduction}
In the literature~\cite[Section~6.4]{abram}, the function
\begin{equation*}
\Gamma(z)=\int_0^{\infty}t^{z-1}\te^{-t}\td t, \quad \Re(z)>0
\end{equation*}
and its logarithmic derivative
$
\psi(z)=[\ln\Gamma(z)]'=\frac{\Gamma'(z)}{\Gamma(z)}
$
are called Euler's gamma function and digamma function respectively. Further, the functions $\psi'(z)$, $\psi''(z)$, $\psi'''(z)$, and $\psi^{(4)}(z)$ are known as the trigamma, tetragamma, pentagamma, and hexagamma functions respectively. All the derivatives $\psi^{(k)}(z)$ for $k\ge0$ are known as polygamma functions.
\par
Recall from Chapter~XIII in~\cite{mpf-1993}, Chapter~1 in~\cite{Schilling-Song-Vondracek-2nd}, and Chapter~IV in~\cite{widder-1941} that, if a function $f(t)$ on an interval $I$ has derivatives of all orders on $I$ and satisfies inequalities $(-1)^{n}f^{(n)}(t)\ge0$ for $t\in I$ and $n\in\{0\}\cup\mathbb{N}$, then we call $f(t)$ a completely monotonic function on $I$. There have been plenty of literature dedicating to the study and applications of completely monotonic functions, logarithmically completely monotonic functions, and completely monotonic degrees. For better motivation of this paper, we would like to recommend three papers~\cite{MR3863624, MR3856139, MR3759696} in 2018, three papers~\cite{MR3961386, MR3952626, MR3997143} in 2019, three papers~\cite{MR4201158, MR4179998, MR4146102} in 2020, three papers~\cite{MR4256259, MR4252699, MR4139119} in 2021, three papers~\cite{MR4403152, MR4376815, MR4366203} in 2022, three papers~\cite{MR4601980, MR4601795, MR4551312} in 2023, and the newly-published expository, survey, and review article~\cite{mathematics-2651178-for-proof.tex} in 2024.
\par
Let $\Phi(x)=x\psi'(x)-1=x\bigl[\psi'(x)-\frac{1}{x}\bigr]$ on $(0,\infty)$. Lemma~2 in~\cite{Alice-AADM-3137.tex} reads that the function $(-1)^k\Phi^{(k)}(x)$ for $k\ge0$ is completely monotonic on $(0,\infty)$.
\par
In~\cite[Theorem~4]{Alice-SPJM.tex} and~\cite[Theorem~4.1]{Alice-Schur-Mon.tex}, Qi turned out the following necessary and sufficient conditions and double inequality:
\begin{enumerate}
\item
if and only if $\alpha\ge2$, the function $\mathfrak{H}_\alpha(x)=\Phi'(x)+\alpha\Phi^2(x)$ is completely monotonic on $(0,\infty)$;
\item
if and only if $\alpha\le1$, the function $-\mathfrak{H}_\alpha(x)$ is completely monotonic on $(0,\infty)$;
\item
the double inequality $-2<\frac{\Phi'(x)}{\Phi^2(x)}<-1$ is valid and sharp in the sense that the lower and upper bounds $-2$ and $-1$ cannot be replaced by any bigger and smaller ones respectively.
\end{enumerate}
In~\cite[Theorem~1.1]{Alice-Schur-Mon.tex}, Qi found the following necessary and sufficient conditions and limits:
\begin{enumerate}
\item
if and only if $\beta\ge2$, the function $H_\beta(x)=\frac{\Phi'(x)}{\Phi^\beta(x)}$ is decreasing on $(0,\infty)$, with the limits
\begin{equation*}
\lim_{x\to0^+}H_\beta(x)=
\begin{cases}
-1, & \beta=2\\
0, &\beta>2
\end{cases}
\quad\text{and}\quad
\lim_{x\to\infty}H_\beta(x)=
\begin{cases}
-2, &\beta=2\\
-\infty, &\beta>2;
\end{cases}
\end{equation*}
\item
if $\beta\le1$, the function $H_\beta(x)$ is increasing on $(0,\infty)$, with the limits
\begin{equation*}
H_\beta(x)\to
\begin{cases}
-\infty, & x\to0^+\\
0, &x\to\infty.
\end{cases}
\end{equation*}
\end{enumerate}
For $k\in\{0\}\cup\mathbb{N}$ and $\lambda_k,\mu_k\in\mathbb{R}$, let
\begin{equation*}
\mathfrak{J}_{k,\lambda_k}(x)=\Phi^{(2k+1)}(x)+\lambda_k\bigl[\Phi^{(k)}(x)\bigr]^2
\quad\text{and}\quad
J_{k,\mu_k}(x)=\frac{\Phi^{(2k+1)}(x)}{\bigl[(-1)^k\Phi^{(k)}(x)\bigr]^{\mu_k}}
\end{equation*}
on $(0,\infty)$.
In~\cite[Theorems~1 and~2]{Alice-AADM-3137.tex}, Qi presented the following necessary and sufficient conditions, limits, and double inequality:
\begin{enumerate}
\item
if and only if $\lambda_k\ge\frac{(2k+2)!}{k!(k+1)!}$, the function $\mathfrak{J}_{k,\lambda_k}(x)$ is completely monotonic on $(0,\infty)$;
\item
if and only if $\lambda_k\le\frac{1}{2}\frac{(2k+2)!}{k!(k+1)!}$, the function $-\mathfrak{J}_{k,\lambda_k}(x)$ is completely monotonic on $(0,\infty)$;
\item
if and only if $\mu_k\ge2$, the function $J_{k,\mu_k}(x)$ is decreasing on $(0,\infty)$, with the limits
\begin{equation*}
\lim_{x\to0^+}J_{k,\mu_k}(x)=
\begin{dcases}
-\frac{1}{2}\frac{(2k+2)!}{k!(k+1)!}, & \mu_k=2\\
0, &\mu_k>2
\end{dcases}
\end{equation*}
and
\begin{equation*}
\lim_{x\to\infty}J_{k,\mu_k}(x)=
\begin{dcases}
-\frac{(2k+2)!}{k!(k+1)!}, &\mu_k=2\\
-\infty, &\mu_k>2;
\end{dcases}
\end{equation*}
\item
if $\mu_k\le1$, the function $J_{k,\mu_k}(x)$ is increasing on $(0,\infty)$, with the limits
\begin{equation*}
J_{k,\mu_k}(x)\to
\begin{cases}
-\infty, & x\to0^+\\
0, &x\to\infty;
\end{cases}
\end{equation*}
\item
the double inequality
\begin{equation*}
-\frac{(2k+2)!}{k!(k+1)!}<\frac{\Phi^{(2k+1)}(x)}{\bigl[\Phi^{(k)}(x)\bigr]^2}<-\frac{1}{2}\frac{(2k+2)!}{k!(k+1)!}
\end{equation*}
is valid on $(0,\infty)$ and sharp in the sense that the lower and upper bounds cannot be replaced by any larger and smaller numbers respectively.
\end{enumerate}
For $k\ge m\ge0$, let
\begin{equation*}
\mathcal{J}_{k,m}(x)=\frac{\Phi^{(2k+2)}(x)}{\Phi^{(k-m)}(x)\Phi^{(k+m+1)}(x)}
\end{equation*}
on $(0,\infty)$. In~\cite[Remark~3]{Alice-AADM-3137.tex}, Qi guessed that the function $\mathcal{J}_{k,m}(x)$ for $k\ge m\ge0$ should be decreasing on $(0,\infty)$ and that the double inequality
\begin{equation}\label{J(0-m-x)-ineq-doub}
-\frac{2(2k+2)!}{k!(k+1)!}<\mathcal{J}_{k,0}(x)<-\frac{(2k+2)!}{k!(k+1)!}
\end{equation}
for $k\ge0$ should be valid on $(0,\infty)$ and sharp in the sense that the lower and upper bounds cannot be replaced by any larger and smaller numbers respectively.
\par
For $m,n\in\{0\}\cup\mathbb{N}$ and $\omega_{m,n}\in\mathbb{R}$, let
\begin{equation}\label{Y(m-n)-dfn}
Y_{m,n}(x)=\frac{\Phi^{(m+n+1)}(x)}{\Phi^{(m)}(x)\Phi^{(n)}(x)}
\end{equation}
and
\begin{equation}\label{Y(m-n)-omega-dfn}
\mathcal{Y}_{m,n;\omega_{m,n}}(x)=\Phi^{(m+n+1)}(x)+\omega_{m,n}\Phi^{(m)}(x)\Phi^{(n)}(x).
\end{equation}
It is clear that
\begin{gather*}
Y_{m,n}(x)=Y_{n,m}(x), \quad \mathcal{Y}_{m,n;\omega_{m,n}}(x)=\mathcal{Y}_{n,m;\omega_{n,m}}(x), \\
Y_{k-m,k+m+1}(x)=\mathcal{J}_{k,m}(x), \quad \mathcal{Y}_{k,k;\omega_{k,k}}(x)=\mathfrak{J}_{k,\omega_{k,k}}(x),\quad
\mathcal{Y}_{0,0;\omega_{0,0}}(x)=\mathfrak{H}_{\omega_{0,0}}(x).
\end{gather*}
In this paper, we will prove decreasing property of the function $Y_{m,n}(x)$ and find necessary and sufficient conditions on $\omega_{m,n}$ for $\pm(-1)^{m+n+1}\mathcal{Y}_{m,n;\omega_{m,n}}(x)$ to be completely monotonic on $(0,\infty)$. These results confirm the above guesses and generalize corresponding ones in~\cite{Alice-SPJM.tex, Alice-Schur-Mon.tex, Alice-AADM-3137.tex} mentioned above.

\section{Lemmas}

The following lemmas are necessary in this paper.

\begin{lem}[{Convolution theorem for the Laplace transforms~\cite[pp.~91--92]{widder-1941}}]\label{convlotion-thm}
Let $f_k(t)$ for $k=1,2$ be piecewise continuous in arbitrary finite intervals included in $(0,\infty)$. If there exist some constants $M_k>0$ and $c_k\ge0$ such that $|f_k(t)|\le M_k \te^{c_kt}$ for $k=1,2$, then
\begin{equation*}
\int_0^\infty\bigg[\int_0^tf_1(u)f_2(t-u)\td u\bigg]\te^{-st}\td t =\int_0^\infty
f_1(u)\te^{-su}\td u\int_0^\infty f_2(v)\te^{-sv}\td v.
\end{equation*}
\end{lem}

\begin{lem}[{\cite[Lemma~4]{Yang-Tian-JIA-2017}}]\label{Laplace-Ratio-Mon}
Let the functions $A(t)$ and $B(t)\ne0$ be defined on $(0,\infty)$ such that their Laplace transforms $\int_{0}^{\infty}A(t)\te^{-xt}\td t$ and $\int_{0}^{\infty}B(t)\te^{-xt}\td t$ exist. If the ratio $\frac{A(t)}{B(t)}$ is increasing, then the ratio $\frac{\int_{0}^{\infty}A(t)\te^{-xt}\td t}{\int_{0}^{\infty}B(t)\te^{-xt}\td t}$ is decreasing on $(0,\infty)$.
\end{lem}

\begin{lem}\label{F(xy)-2-power-lem}
Let $x,y\in\mathbb{R}$ such that $0<2x<y$.
\begin{enumerate}
\item
When $y>2x>2\bigl(2+\frac{1}{\ln2}\bigr)=6.885390\dotsc$, the function
\begin{equation*}
F(x,y)=2\biggl(\frac{1}{x}-\frac{1}{y-x}\biggr) +\frac{1}{2}\biggl(\frac{2^{y-x}}{y-x}-\frac{2^x}{x}\biggr) -\frac{2^{y-x}-2^{x}}{(y-x)x}
\end{equation*}
is positive.
\item
For $k,m\in\mathbb{N}$ such that $6\le2m<k$, the sequence $F(m,k)$ is positive.
\end{enumerate}
\end{lem}

\begin{proof}
The function $F(x,y)$ can be rearranged as
\begin{equation*}
F(x,y)=\frac{2(y-2x)+2^{x-1}\bigl[2-y+x+(x-2)2^{y-2x}\bigr]}{x(y-x)}.
\end{equation*}
Therefore, it suffices to prove $2-y+x+(x-2)2^{y-2x}>0$, that is,
\begin{equation}\label{equiv-ineq-2power}
2^{y-2x}>\frac{y-x-2}{x-2}.
\end{equation}
Replacing $y-2x$ by $t$ in~\eqref{equiv-ineq-2power} leads to
\begin{equation}\label{equiv-in-per}
2^t>\frac{t+x-2}{x-2}=1+\frac{t}{x-2}
\end{equation}
for $t>0$ and $x>2$. The inequality~\eqref{equiv-in-per} can be reformulated as $x>2+\frac{t}{2^t-1}$.
Since the function $\frac{t}{2^t-1}$ is decreasing from $(0,\infty)$ onto $\bigl(0,\frac{1}{\ln2}\bigr)$, it is sufficient for $x>2+\frac{1}{\ln2}=3.442695\dotsc$.
\par
Repeating those arguments before the inequality~\eqref{equiv-ineq-2power} hints us that, for proving $F(m,k)>0$, it is sufficient to show
\begin{equation*}
2^{k-2m}>\frac{k-m-2}{m-2}=1+\frac{k-2m}{m-2}
\end{equation*}
which can be rewritten as
\begin{equation}\label{equ-ineq-k-m-er}
\frac{k-2m}{2^{k-2m}-1}<m-2
\end{equation}
Since $\frac{t}{2^t-1}$ is decreasing in $t\in(0,\infty)$ and $k-2m\ge1$, the largest value of the left hand side in the inequality~\eqref{equ-ineq-k-m-er} is $\frac{1}{2^1-1}=1$ which means that the strict inequality~\eqref{equ-ineq-k-m-er} is valid for all $m\ge4$.
As a result, the sequence $F(m,k)$ is positive for all $m\ge4$.
\par
When $m=3$, the sequence $F(3,k)$ is
\begin{equation*}
F(3,k)=\frac{2^k-32k+128}{48(k-3)}
=\frac{2^5\bigl[2^{k-5}-(k-4)\bigr]}{48(k-3)}
\end{equation*}
which is positive for all $k>2\cdot3=6$.
The proof of Lemma~\ref{F(xy)-2-power-lem} is complete.
\end{proof}

\begin{lem}\label{h(t)-ratio-mon-lem}
Let
\begin{equation*}
h(t)=\begin{dcases}
\frac{\te^t(\te^t-1-t)}{(\te^t-1)^2}, & t\ne0\\
\frac{1}{2}, & t=0
\end{dcases}
\end{equation*}
on $(-\infty,\infty)$. Then, for any fixed $s\in(0,1)$, the ratio $\frac{h(st)}{h^s(t)}$ is increasing in $t$ from $(0,\infty)$ onto $\bigl(\frac{1}{2^{1-s}},1\bigr)$.
\end{lem}

\begin{proof}
It is easy to see that
\begin{equation*}
\lim_{t\to0}\frac{h(st)}{h^s(t)}=\frac{\lim_{t\to0}h(st)}{\lim_{t\to0}h^s(t)}=\frac{\frac{1}{2}}{\frac{1}{2^s}}=\frac{1}{2^{1-s}}
\end{equation*}
and
\begin{equation*}
\lim_{t\to\infty}H_s(t)=\frac{\lim_{t\to\infty}h(st)}{\lim_{t\to0}h^s(t)}=\frac{1}{1^s}=1.
\end{equation*}
\par
Direct differentiating and expanding to power series give
\begin{gather*}
\frac{\td}{\td t}\biggl[\frac{h(st)}{h^s(t)}\biggr]
=-\frac{s\te^{(1+s)t}\begin{bmatrix}
(t-2)\te^{(1+2s)t}+(t+2)\te^{2st}+(2-st)\te^{(2+s)t}\\
+4(s-1)t\te^{(1+s)t}-(2st^2+3st+2)\te^{st}\\
-(st+2)\te^{2t}+\bigl(2st^2+3t+2\bigr)\te^t+(s-1)t\end{bmatrix}} {(\te^t-1)^3(\te^{st}-1)^3h^{s+1}(t)}\\
=\frac{s\te^{(1+s)t}\sum\limits_{k=7}^{\infty}
\begin{pmatrix}
(3k+2)s^k+2(2s+1)^k+ks(s+2)^{k-1}\\
+2k^2s^{k-1}+4k(s+1)^{k-1}\\
+2k\bigl(1+2^{k-2}\bigr)s+2^{k+1}\\
-\bigl[2^{k+1}s^k+2(s+2)^k\\
+4ks(s+1)^{k-1}+2k\bigl(1+2^{k-2}\bigr)s^{k-1}\\
+k(2s+1)^{k-1}+2k^2s+3k+2\bigr]
\end{pmatrix}\dfrac{t^k}{k!}}{(\te^t-1)^3(\te^{st}-1)^3h^{s+1}(t)}\\
=\frac{s\te^{(1+s)t}\sum\limits_{k=7}^{\infty}
\begin{pmatrix}
\sum\limits_{m=1}^{k-1}\bigl[k2^{k-m}\binom{k-1}{m-1}+4k\binom{k-1}{m}+2^{m+1}\binom{k}{m}\bigr]s^m\\
+2k\bigl(1-k+2^{k-2}\bigr)\bigl(s-s^{k-1}\bigr)\\
-\sum\limits_{m=1}^{k-1}\big[k2^m\binom{k-1}{m}+4k\binom{k-1}{m-1}+2^{k-m+1}\binom{k}{m}\bigr]s^m
\end{pmatrix}\dfrac{t^k}{k!}}{(\te^t-1)^3(\te^{st}-1)^3h^{s+1}(t)}\\
=\frac{s\te^{(1+s)t}\sum\limits_{k=7}^{\infty}
\begin{pmatrix}
\sum\limits_{m=1}^{k-1}\bigl[k2^{k-m}\binom{k-1}{m-1}+4k\binom{k-1}{m}+2^{m+1}\binom{k}{m}\bigr]\bigl(s^m-s^{k-m}\bigr)\\
+2k\bigl(1-k+2^{k-2}\bigr)\bigl(s-s^{k-1}\bigr)
\end{pmatrix}\dfrac{t^k}{k!}}{(\te^t-1)^3(\te^{st}-1)^3h^{s+1}(t)}\\
=\frac{s\te^{(1+s)t}\sum\limits_{k=7}^{\infty}
\begin{pmatrix}
\sum\limits_{3\le m<\frac{k}{2}}\bigl[k2^{k-m}\binom{k-1}{m-1}-k2^{m}\binom{k-1}{k-m-1}\\
+4k\binom{k-1}{m}-4k\binom{k-1}{k-m}\\
+2^{m+1}\binom{k}{m}-2^{k-m+1}\binom{k}{k-m}\bigr]\bigl(s^m-s^{k-m}\bigr)
\end{pmatrix}\dfrac{t^k}{k!}}{(\te^t-1)^3(\te^{st}-1)^3h^{s+1}(t)}\\
=\frac{s\te^{(1+s)t}\sum\limits_{k=7}^{\infty}
\begin{pmatrix}
2\sum\limits_{3\le m<\frac{k}{2}}\frac{k!}{(m-1)! (k-m-1)!}\bigl[2\bigl(\frac{1}{m}-\frac{1}{k-m}\bigr)\\
+\frac{1}{2}\bigl(\frac{2^{k-m}}{k-m}-\frac{2^m}{m}\bigr) -\frac{2^{k-m}-2^{m}}{(k-m)m}\bigr]\bigl(s^m-s^{k-m}\bigr)
\end{pmatrix}\dfrac{t^k}{k!}}{(\te^t-1)^3(\te^{st}-1)^3h^{s+1}(t)}\\
=\frac{s^4\te^{(1+s)t}
\begin{pmatrix}
\frac{1-s}{36}t^7+\frac{1-s^2}{45}t^8 +\frac{22(1-s^3)+15(s-s^2)}{2160}t^9\\
+\frac{52(1-s^4)+63(s-s^3)}{15120}t^{10}\\
+\frac{285(1-s^5)+470(s-s^4)+238(s^2-s^3)}{302400} t^{11}+\dotsm
\end{pmatrix}}{(\te^t-1)^3(\te^{st}-1)^3h^{s+1}(t)}.
\end{gather*}
Utilizing Lemma~\ref{F(xy)-2-power-lem} reveals that the derivative $\frac{\td}{\td t}\bigl[\frac{h(st)}{h^s(t)}\bigr]$ is positive for $s\in(0,1)$ and $t>0$. Consequently, for $s\in(0,1)$, the ratio $\frac{h(st)}{h^s(t)}$ is increasing in $t>0$.
The proof of Lemma~\ref{h(t)-ratio-mon-lem} is complete.
\end{proof}

\begin{lem}[{\cite[Lemma~2]{Alice-AADM-3137.tex}}]\label{Alice-PolyG-4limits}
For $k\ge0$, the function $(-1)^k\Phi^{(k)}(x)$ is completely monotonic on $(0,\infty)$, with the limits
\begin{equation}\label{lim-to0-to-infnity-unit}
(-1)^{k}x^{k+1}\Phi^{(k)}(x)
\to
\begin{dcases}
k!, & x\to0^+;\\
\frac{k!}{2}, & x\to\infty.
\end{dcases}
\end{equation}
\end{lem}

\begin{lem}[{Bernstein's theorem~\cite[p.~161, Theorem~12b]{widder-1941}}]\label{Theorem12b-widder-p161}
A function $f(x)$ is completely monotonic on $(0,\infty)$ if and only if
\begin{equation}\label{Theorem12b-Laplace}
f(x)=\int_0^\infty \te^{-xt}\td\sigma(t), \quad x\in(0,\infty),
\end{equation}
where $\sigma(s)$ is non-decreasing and the integral in~\eqref{Theorem12b-Laplace} converges for $x\in(0,\infty)$.
\end{lem}

\section{Decreasing property}

In this section, we prove that the function $Y_{m,n}(x)$ defined in~\eqref{Y(m-n)-dfn} is decreasing.

\begin{thm}\label{J(k-m-x)-decr-mon-thm}
For $m,n\in\{0\}\cup\mathbb{N}$, the function $Y_{m,n}(x)$ defined in~\eqref{Y(m-n)-dfn} is decreasing in $x$ from $(0,\infty)$ onto the interval $\bigl(-\frac{2(m+n+1)!}{m!n!},-\frac{(m+n+1)!}{m!n!}\bigr)$. Consequently, for $m,n\in\{0\}\cup\mathbb{N}$, the double inequality
\begin{equation}\label{J(m-m-x)-ineq-doub}
-\frac{2(m+n+1)!}{m!n!}<Y_{m,n}(x)<-\frac{(m+n+1)!}{m!n!}
\end{equation}
is valid on $(0,\infty)$ and sharp in the sense that the lower and upper bounds cannot be replaced by any larger and smaller numbers respectively.
\end{thm}

\begin{proof}
In the proof of~\cite[Theorem~4]{Alice-SPJM.tex}, Qi established that
\begin{equation}\label{x-psi-prime-d-int}
\Phi(x)=\int_{0}^{\infty}h(t)\te^{-xt}\td t.
\end{equation}
Then the ratio $Y_{m,n}(x)$ can be rewritten as
\begin{align*}
Y_{m,n}(x)&=-\frac{\int_{0}^{\infty}t^{m+n+1}h(t)\te^{-xt}\td t}{\int_{0}^{\infty}t^{m}h(t)\te^{-xt}\td t \int_{0}^{\infty}t^{n}h(t)\te^{-xt}\td t}\\
&=-\frac{\int_{0}^{\infty}t^{m+n+1}h(t)\te^{-xt}\td t}
{\int_{0}^{\infty}\bigl[\int_{0}^{t}u^{m}(t-u)^{n}h(u)h(t-u)\td u\bigr]\te^{-xt}\td t},
\end{align*}
where we used Lemma~\ref{convlotion-thm}. Basing on Lemma~\ref{Laplace-Ratio-Mon}, in order to prove decreasing property of $Y_{m,n}(x)$, it suffices to show that the ratio
\begin{equation}\label{mathfrak-Y-m-n(t)}
\mathfrak{Y}_{m,n}(t)=\frac{t^{m+n+1}h(t)}{\int_{0}^{t}u^{m}(t-u)^{n}h(u)h(t-u)\td u}
\end{equation}
is decreasing in $t\in(0,\infty)$. By changing the variable $u=\frac{1+v}{2}t$, the denominator of $\mathfrak{Y}_{m,n}(t)$ becomes
\begin{equation*}
\biggl(\frac{t}{2}\biggr)^{m+n+1} \int_{-1}^{1}(1+v)^{m}(1-v)^{n} h\biggl(\frac{1+v}{2}t\biggr)h\biggl(\frac{1-v}{2}t\biggr)\td v.
\end{equation*}
Accordingly, we obtain
\begin{equation}\label{recipr-mathfrak-Y-m-n(t)}
\begin{aligned}
\frac{1}{\mathfrak{Y}_{m,n}(t)}&=\frac{\int_{-1}^{1}(1+v)^{m}(1-v)^{n} h\bigl(\frac{1+v}{2}t\bigr)h\bigl(\frac{1-v}{2}t\bigr)\td v}{2^{m+n+1}h(t)}\\
&=\frac{1}{2^{m+n+1}}\int_{-1}^{1}(1+v)^{m}(1-v)^{n} \frac{h\bigl(\frac{1+v}{2}t\bigr)h\bigl(\frac{1-v}{2}t\bigr)}{h(t)}\td v\\
&=\frac{1}{2^{m+n+1}}\int_{-1}^{1}(1+v)^{m}(1-v)^{n} \frac{h(st)}{h^s(t)}\frac{h((1-s)t)}{h^{1-s}(t)}\td v,
\end{aligned}
\end{equation}
where $s=\frac{1+v}{2}\in(0,1)$. From Lemma~\ref{h(t)-ratio-mon-lem}, we find that the function
$
\frac{h(st)}{h^s(t)}\frac{h((1-s)t)}{h^{1-s}(t)}
$
is increasing in $t\in(0,\infty)$ for any fixed $s\in(0,1)$. Hence, the function $\mathfrak{Y}_{m,n}(t)$ is decreasing on $(0,\infty)$. Therefore, the function $Y_{m,n}(x)$ for $m,n\in\{0\}\cup\mathbb{N}$ is decreasing on $(0,\infty)$.
\par
Making use of the limits in~\eqref{lim-to0-to-infnity-unit} in Lemma~\ref{Alice-PolyG-4limits} yields
\begin{align*}
Y_{m,n}(x)&=-\frac{(-1)^{m+n+1}x^{m+n+2}\Phi^{(m+n+1)}(x)}{[(-1)^{m}x^{m+1}\Phi^{(m)}(x)] [(-1)^{n}x^{k+m+2}\Phi^{(n)}(x)]}\\
&\to
\begin{dcases}
-\frac{(m+n+1)!}{m!n!}, & x\to0^+;\\
-\frac{2(m+n+1)!}{m!n!}, & x\to\infty.
\end{dcases}
\end{align*}
The proof of Theorem~\ref{J(k-m-x)-decr-mon-thm} is complete.
\end{proof}

\section{Necessary and sufficient conditions of complete monotonicity}
In this section, we discover necessary and sufficient conditions on $\omega_{m,n}$ for the function $\pm(-1)^{m+n+1}\mathcal{Y}_{m,n;\omega_{m,n}}(x)$ defined in~\eqref{Y(m-n)-omega-dfn} to be completely monotonic.

\begin{thm}\label{NSC-omega-thm}
For $m,n\in\{0\}\cup\mathbb{N}$ and $\omega_{m,n}\in\mathbb{R}$,
\begin{enumerate}
\item
if and only if $\omega_{m,n}\le\frac{(m+n+1)!}{m!n!}$, the function $(-1)^{m+n+1}\mathcal{Y}_{m,n;\omega_{m,n}}(x)$ is completely monotonic on $(0,\infty)$;
\item
if and only if $\omega_{m,n}\ge\frac{2(m+n+1)!}{m!n!}$, the function $(-1)^{m+n}\mathcal{Y}_{m,n;\omega_{m,n}}(x)$ is completely monotonic on $(0,\infty)$;
\item
the double inequality~\eqref{J(m-m-x)-ineq-doub} is valid on $(0,\infty)$ and sharp in the sense that the lower and upper bounds cannot be replaced by any larger and smaller numbers respectively.
\end{enumerate}
\end{thm}

\begin{proof}
As done in the proof of Theorem~\ref{J(k-m-x)-decr-mon-thm}, by virtue of the integral representation~\eqref{x-psi-prime-d-int} and Lemma~\ref{convlotion-thm}, we acquire
\begin{gather*}
\begin{aligned}
(-1)^{m+n+1}\mathcal{Y}_{m,n;\omega_{m,n}}(x)&=\biggl[\int_{0}^{\infty}t^{m+n+1}h(t)\te^{-xt}\td t\\
&\quad-\omega_{m,n}\int_{0}^{\infty}t^{m}h(t)\te^{-xt}\td t\int_{0}^{\infty}t^{n}h(t)\te^{-xt}\td t\biggr]
\end{aligned}\\
\begin{aligned}
&=\int_{0}^{\infty}\biggl[t^{m+n+1}h(t)-\omega_{m,n}\int_{0}^{t}u^{m}(t-u)^{n}h(u)h(t-u)\td u\biggr]\te^{-xt}\td t\\
&=\int_{0}^{\infty}\biggl[1-\frac{\omega_{m,n}}{\mathfrak{Y}_{m,n}(t)}\biggr]t^{m+n+1}h(t)\te^{-xt}\td t,
\end{aligned}
\end{gather*}
where $\mathfrak{Y}_{m,n}(t)$ is defined by~\eqref{mathfrak-Y-m-n(t)} and it has been proved in the proof of Theorem~\ref{J(k-m-x)-decr-mon-thm} to be decreasing on $(0,\infty)$.
From Lemma~\ref{h(t)-ratio-mon-lem}, we conclude that the function $\frac{h(st)}{h^s(t)}\frac{h((1-s)t)}{h^{1-s}(t)}$ is increasing in $t$ from $(0,\infty)$ onto $\bigl(\frac{1}{2},1\bigr)$. Accordingly, by virtue of~\eqref{recipr-mathfrak-Y-m-n(t)}, we arrive at the sharp inequalities
\begin{equation*}
\frac{1}{2^{m+n+2}}\int_{-1}^{1}(1+v)^{m}(1-v)^{n}\td v <\frac{1}{\mathfrak{Y}_{m,n}(t)} <\frac{1}{2^{m+n+1}}\int_{-1}^{1}(1+v)^{m}(1-v)^{n}\td v.
\end{equation*}
Since
\begin{align*}
\int_{-1}^{1}(1+v)^{m}(1-v)^{n}\td v&=\int_{0}^{1}\bigl[(1+v)^{m}(1-v)^{n}+(1-v)^{m}(1+v)^{n}\bigr]\td v\\
&=2^{m+n+1}B(m+1,n+1)\\
&=2^{m+n+1}\frac{m!n!}{(m+n+1)!},
\end{align*}
where we used the formula
\begin{align*}
\int_{0}^{1}\bigl[(1+x)^{\mu-1}(1-x)^{\nu-1}+(1+x)^{\nu-1}(1-x)^{\mu-1}\bigr]\td x
&=2^{\mu+\nu-1}B(\mu,\nu)\\
&=2^{\mu+\nu-1}\frac{\Gamma(\mu)\Gamma(\nu)}{\Gamma(\mu+\nu)}
\end{align*}
for $\Re(\mu),\Re(\nu)>0$ in~\cite[p.~321, 3.214]{Gradshteyn-Ryzhik-Table-8th}, the double inequality
\begin{equation*}
\frac{1}{2}\frac{m!n!}{(m+n+1)!}<\frac{1}{\mathfrak{Y}_{m,n}(t)}
<\frac{m!n!}{(m+n+1)!}
\end{equation*}
is valid and sharp on $(0,\infty)$. Consequently, by virtue of Lemma~\ref{Theorem12b-widder-p161}, if and only if $\omega_{m,n}\le\frac{(m+n+1)!}{m!n!}$, the function $(-1)^{m+n+1}\mathcal{Y}_{m,n;\omega_{m,n}}(x)$ is completely monotonic on $(0,\infty)$; if and only if $\omega_{m,n}\ge\frac{2(m+n+1)!}{m!n!}$, the function $(-1)^{m+n}\mathcal{Y}_{m,n;\omega_{m,n}}(x)$ is completely monotonic on $(0,\infty)$.
\par
The double inequality~\eqref{J(m-m-x)-ineq-doub} follows from complete monotonicity of the functions $\pm(-1)^{m+n+1}\mathcal{Y}_{m,n;\omega_{m,n}}(x)$. The proof of the sharpness of the double inequality~\eqref{J(m-m-x)-ineq-doub} is the same as done in the proof of Theorem~\ref{J(k-m-x)-decr-mon-thm}.
The proof of Theorem~\ref{NSC-omega-thm} is complete.
\end{proof}

\section{Remarks}

In this section, we list several remarks related to our main results and their proofs in this paper.

\begin{rem}
Lemma~\ref{h(t)-ratio-mon-lem} in this paper generalizes a conclusion in~\cite[Lemma~2.3]{Alice-Schur-Mon.tex} which reads that the function $\frac{h(2t)}{h^2(t)}$ is decreasing from $(0,\infty)$ onto $(1,2)$.
\end{rem}

\begin{rem}
The function $F(x,y)$ discussed in Lemma~\ref{F(xy)-2-power-lem} can be reformulated as
\begin{equation*}
F(x,y)=\biggl(\frac{1}{x}-\frac{1}{y-x}\biggr)\Biggl[2-\frac{1}{2}\frac{\frac{2^{y-x}}{y-x}-\frac{2^x}{x}} {\frac{1}{y-x}-\frac{1}{x}} -\frac{2^{y-x}-2^{x}}{(y-x)-x}\Biggr],
\end{equation*}
in which the functions
\begin{equation*}
\frac{\frac{2^{y-x}}{y-x}-\frac{2^x}{x}} {\frac{1}{y-x}-\frac{1}{x}}
\quad\text{and}\quad
\frac{2^{y-x}-2^{x}}{(y-x)-x}
\end{equation*}
can be regarded as special means~\cite{bullenmean, Recipr-Sqrt-Geom-S.tex}.
\par
Let $x,y\in\mathbb{R}$ such that $0<x<\frac{y}{2}$. Motivated by Lemma~\ref{F(xy)-2-power-lem}, we guess that,
\begin{enumerate}
\item
when $2<x<\frac{y}{2}$, the function $F(x,y)$ is positive;
\item
when $y>4$ and $0<x<2$, the function $F(x,y)$ is negative.
\end{enumerate}
Furthermore, one can discuss positivity and negativity of the function $F(x,y)$ for all $x,y$ satisfying $0<x<\frac{y}{2}$.
\end{rem}

\begin{rem}
When taking $m=k$ and $n=k+1$, the double inequality~\eqref{J(m-m-x)-ineq-doub} in Theorem~\ref{J(k-m-x)-decr-mon-thm} becomes the double inequality~\eqref{J(0-m-x)-ineq-doub} guessed by the corresponding and third author in~\cite[Remark~3]{Alice-AADM-3137.tex}.
\end{rem}

\begin{rem}
For $m,n\in\{0\}\cup\mathbb{N}$, direct differentiation gives
\begin{equation*}
Y_{m,n}'(x)=\frac{\Phi^{(m+n+2)}(x)\bigl[\Phi^{(m)}(x)\Phi^{(n)}(x)\bigr] -\Phi^{(m+n+1)}(x)\bigl[\Phi^{(m)}(x)\Phi^{(n)}(x)\bigr]'} {[\Phi^{(m)}(x)\Phi^{(n)}(x)]^2}
\end{equation*}
on $(0,\infty)$. The decreasing monotonicity of $Y_{m,n}(x)$ in Theorem~\ref{J(k-m-x)-decr-mon-thm} implies that, for $m,n\in\{0\}\cup\mathbb{N}$, the inequality
\begin{equation*}
\Phi^{(m+n+1)}(x)\bigl[\Phi^{(m)}(x)\Phi^{(n)}(x)\bigr]'>\Phi^{(m+n+2)}(x)\bigl[\Phi^{(m)}(x)\Phi^{(n)}(x)\bigr],
\end{equation*}
equivalently,
\begin{equation*}
\frac{\bigl[\Phi^{(m)}(x)\Phi^{(n)}(x)\bigr]'}{\Phi^{(m)}(x)\Phi^{(n)}(x)}<\frac{\Phi^{(m+n+2)}(x)}{\Phi^{(m+n+1)}(x)},
\end{equation*}
is valid on $(0,\infty)$.
\par
We guess that, for $m,n\in\{0\}\cup\mathbb{N}$, the function
\begin{equation*}
\Phi^{(m+n+1)}(x)\bigl[\Phi^{(m)}(x)\Phi^{(n)}(x)\bigr]'-\Phi^{(m+n+2)}(x)\bigl[\Phi^{(m)}(x)\Phi^{(n)}(x)\bigr]
\end{equation*}
is completely monotonic in $x\in(0,\infty)$.
\par
One can also consider necessary and sufficient conditions on $\Lambda_{m,n}\in\mathbb{R}$ for $m,n\in\{0\}\cup\mathbb{N}$ such that the function
\begin{equation*}
\Phi^{(m+n+1)}(x)\bigl[\Phi^{(m)}(x)\Phi^{(n)}(x)\bigr]'-\Lambda_{m,n}\Phi^{(m+n+2)}(x)\bigl[\Phi^{(m)}(x)\Phi^{(n)}(x)\bigr]
\end{equation*}
and its opposite are respectively completely monotonic on $(0,\infty)$.
\end{rem}

\begin{rem}
This paper is a revised version of the electronic preprint~\cite{Alice-PolyG-Mon-Conj.tex}, was reported online by Qi between 09:30--09:45 on 21 November 2020 and on the 3rd International Conference on Mathematical and Related Sciences: Current Trends and Developments (ICMRS 2020) held in Turkey, and is the sixth one in a series of works including the articles~\cite{Alice-SPJM.tex, Alice-y-x-curvature.tex, Alice-Phi-CM.tex, Alice-Schur-Mon.tex, Alice-AADM-3137.tex, Alice-y-x-conj-one.tex, Slovaca-6019.tex, Alice-y-x-conj-two.tex, TWMS-20657.tex, Alice-y-x-curv-notes.tex} and the review article~\cite{mathematics-2651178-for-proof.tex}.
\end{rem}

\section{Authors' statements}

\paragraph{\bf Conflict of interest}
The authors declare that they have no any conflict of competing interests.

\paragraph{\bf Funding information}
Not applicable.

\paragraph{\bf Author contribution}
All authors contributed equally to the manuscript and read and approved the final manuscript.

\paragraph{\bf Data availability statement}
Data sharing is not applicable to this article as no new data were created or analyzed in this study.

\paragraph{\bf Ethical approval}
The conducted research is not related to either human or animal use.

\paragraph{\bf Informed consent}
Not applicable.

\paragraph{\bf Acknowledgements}
The authors appreciate anonymous referees for their careful corrections to, helpful suggestions to, and valuable comments on the original version of this paper.

%

\end{document}